\documentclass[12pt]{article}
\usepackage{amsmath, amsthm, amssymb}
\usepackage{tkz-graph}
\usepackage{marginnote}
\usepackage{verbatim}
\usepackage[top=1.0in, bottom=1.0in, left=1.0in, right=1.0in]{geometry}
\usepackage{color}
\usepackage[backref=page]{hyperref}

\usepackage{sectsty}
\allsectionsfont{\sffamily}

\makeatletter
\newtheorem*{rep@theorem}{\rep@title}
\newcommand{\newreptheorem}[2]{
\newenvironment{rep#1}[1]{
 \def\rep@title{#2 \ref{##1}}
 \begin{rep@theorem}}
 {\end{rep@theorem}}}
\makeatother

\theoremstyle{plain}
\newtheorem{thm}{Theorem}
\newreptheorem{thm}{Theorem}

\newreptheorem{prop}{Proposition}
\newtheorem{lem}[thm]{Lemma}
\newreptheorem{lem}{Lemma}

\newreptheorem{conjecture}{Conjecture}

\newreptheorem{cor}{Corollary}

\newtheorem*{KernelMagic}{Kernel Magic}

\newtheorem*{MainTheorem}{Main Theorem}

\newtheorem*{TheConjecture}{Conjecture}

\theoremstyle{definition}
\newtheorem*{TheDefinition}{Definition}

\theoremstyle{remark}

\newcommand{\fancy}[1]{\mathcal{#1}}
\newcommand{\C}[1]{\fancy{C}_{#1}}

\newcommand{\IN}{\mathbb{N}}

\renewcommand{\L}{\fancy{L}}
\newcommand{\HH}{\fancy{H}}

\newcommand{\set}[1]{\left\{ #1 \right\}}

\newcommand{\card}[1]{\left|#1\right|}
\newcommand{\size}[1]{\left\Vert#1\right\Vert}

\newcommand{\parens}[1]{\left( #1 \right)}

\newcommand{\DefinedAs}{\mathrel{\mathop:}=}

\newcommand{\mic}{\operatorname{mic}}

\newcommand\restr[2]{{
  \left.\kern-\nulldelimiterspace 
  #1 
  \vphantom{\big|} 
  \right|_{#2} 
  }}

\def\C{\fancy{C}}

\title{A better lower bound on average degree of 4-list-critical graphs}
\author{Landon Rabern}

\begin{document}
\maketitle
\begin{abstract}
		This short note proves that every incomplete $k$-list-critical graph has average degree at least $k-1 + \frac{k-3}{k^2-2k+2}$.  This improves the best known bound for $k = 4,5,6$.
		The same bound holds for online $k$-list-critical graphs.
\end{abstract}

\section{Introduction}
A graph $G$ is \emph{$k$-list-critical} if $G$ is not $(k-1)$-choosable, but every
proper subgraph of $G$ is $(k-1)$-choosable.  For further definitions and notation, see \cite{OreVizing, DischargingLowerBound}. 
Table \ref{TheTable} shows some history of lower bounds on the average degree of $k$-list-critical graphs.

\begin{MainTheorem}
	Every incomplete $k$-list-critical graph has average degree at least \[k-1 + \frac{k-3}{k^2-2k+2}.\]
\end{MainTheorem}

Main Theorem gives a lower bound of $3 + \frac{1}{10}$ for $4$-list-critical graphs. This is the first improvement over Gallai's bound of $3 + \frac{1}{13}$. 
The same proof shows that Main Theorem holds for online $k$-list-critical graphs as well.  Our primary tool is a lemma proved with Kierstead \cite{KernelMagic} 
that generalizes a kernel technique of Kostochka and Yancey \cite{kostochkayancey2012ore}.

\begin{TheDefinition} The \emph{maximum independent cover number }of a graph $G$
	is the maximum $\mic(G)$ of $\size{I, V(G) \setminus I}$ over all independent sets $I$
	of $G$. 
\end{TheDefinition}

\begin{KernelMagic}[Kierstead and R. \cite{KernelMagic}]\label{ConsantListMicStrength} 
	Every $k$-list-critical graph $G$ satisfies
	\[2\size{G} \ge (k-2)\card{G} + \mic(G) + 1.\]
\end{KernelMagic}
The previous best bounds in Table \ref{TheTable} for $k$-list-critical graphs hold for k-Alon-Tarsi-critical graphs as well. Since Kernel Magic relies on the Kernel Lemma, 
our proof does not work for k-Alon-Tarsi-critical graphs.  Any improvement over Gallai's bound of $3 + \frac{1}{13}$ for 4-Alon-Tarsi-critical graphs would be interesting.

\begin{table}
	\begin{center}
		\begin{tabular}{|c|c|c|c|c|c|c|c|c|}
			\hline
			&\multicolumn{4}{ |c| }{$k$-Critical
				$G$}&\multicolumn{4}{|c|}{$k$-List Critical $G$}\\
			\hline
			& Gallai \cite{gallai1963kritische}
			& Kriv \cite{krivelevich1997minimal}
			& KS \cite{kostochkastiebitzedgesincriticalgraph}
			& KY \cite{kostochkayancey2012ore}
			& KS \cite{kostochkastiebitzedgesincriticalgraph} 
			& KR \cite{OreVizing}
			& CR \cite{DischargingLowerBound}
			& Here \\
			$k$ & $d(G) \ge$ & $d(G) \ge$ & $d(G) \ge$ & $d(G) \ge$ & $d(G) \ge$ & $d(G) \ge$ & $d(G) \ge$ & $d(G) \ge$\\
			\hline 
			4 & 3.0769 &3.1429&---&3.3333& --- & --- & --- & \bf{3.1000}\\
			5 & 4.0909 &4.1429&---&4.5000& --- & 4.0984 & 4.1000 & \bf{4.1176}\\
			6 & 5.0909 &5.1304&5.0976&5.6000& --- & 5.1053 & 5.1076 & \bf{5.1153}\\
			7 & 6.0870 &6.1176&6.0990&6.6667& --- & 6.1149 & \bf{6.1192} & 6.1081\\
			8 & 7.0820 &7.1064&7.0980&7.7143& --- & 7.1128 & \bf{7.1167} & 7.1000\\
			9 & 8.0769 &8.0968&8.0959&8.7500& 8.0838 & 8.1094 & \bf{8.1130} & 8.0923\\
			10 & 9.0722 &9.0886&9.0932&9.7778& 9.0793 & 9.1055 & \bf{9.1088} & 9.0853\\
			15 & 14.0541 &14.0618&14.0785&14.8571& 14.0610 & 14.0864 & \bf{14.0884} & 14.0609\\
			20 & 19.0428 &19.0474&19.0666&19.8947& 19.0490 & 19.0719 & \bf{19.0733} & 19.0469 \\
			\hline
		\end{tabular}
	\end{center}
	\caption{History of lower bounds on the average degree $d(G)$ of $k$-critical and $k$-list-critical graphs $G$.}
	\label{TheTable}
\end{table}

\section{The Proof}
The connected graphs in which each block is a complete graph
or an odd cycle are called \emph{Gallai trees}.  Gallai \cite{gallai1963kritische} proved that in a $k$-critical graph, the vertices of degree $k-1$ induce a disjoint union of Gallai trees.  The same is true for $k$-list-critical graphs (\cite{borodin1977criterion, erdos1979choosability}).  For a graph $T$ and $k \in \IN$, let $\beta_k(T)$ be the independence number of the subgraph of $T$ induced on the vertices of degree $k-1$.  When $k$ is defined in the context, put $\beta(T) 
\DefinedAs \beta_k(T)$.  

\begin{lem}\label{BetaaaBound}
	If $k \ge 4$ and $T \ne K_k$ is a Gallai tree with maximum degree at most $k-1$, then
	\[2||T|| \le (k-2)|T| + 2\beta(T).\]
\end{lem}
\begin{proof}
	Suppose the lemma is false and choose a counterexample $T$ minimizing $\card{T}$.  Plainly, $T$ has more than one block.  Let $A$ be an endblock of $T$ and let $x$ be the unique cutvertex of $T$ with $x \in V(A)$.
	Consider $T' \DefinedAs T - (V(A)\setminus\set{x})$.  By minimality of $\card{T}$,
	\begin{equation*}
		2\size{T} - 2\size{A} \le (k-2)(\card{T} + 1 - \card{A}) + 2\beta(T').
	\end{equation*}
		Since $T$ is a counterexample, $2\size{A} > (k-2)(\card{A} - 1)$.  So, if $k > 4$, then $A = K_{k-1}$ and if $k=4$, then $A$ is an odd cycle.  In both cases, $d_T(x) = k-1$.
	Consider $T^* \DefinedAs T - V(A)$.  By minimality of $\card{T}$,
	\begin{equation*}
	2\size{T} - 2\size{A} - 2 \le (k-2)(\card{T} - \card{A}) + 2\beta(T^*).
	\end{equation*}
	Since $T$ is a counterexample, $2\size{A} + 2 > (k-2)\card{A} + 2(\beta(T) - \beta(T^*))$.  In $T^*$, all of $x$'s neighbors have degree at most $k-2$.
	But $d_T(x) = k-1$, so some vertex in $\set{x} \cup N(x)$ is in a maximum independent set of degree $k-1$ vertices in $T$.  Hence $\beta(T^*) \le \beta(T) - 1$, which gives
	\begin{equation*}
	 2\size{A} > (k-2)\card{A},
	\end{equation*}
	a contradiction since $k \ge 4$.
\end{proof}

\begin{proof}[Proof of Main Theorem]
Let $G \ne K_k$ be a $k$-list-critical graph.  The theorem is trivially true if $k \le 3$, so suppose $k \ge 4$. Let $\L \subseteq V(G)$ be the vertices with degree $k-1$ and let $\HH = V(G) \setminus \L$.  Put $\size{\L} \DefinedAs \size{G[\L]}$ and $\size{\HH} \DefinedAs \size{G[\HH]}$.  
By Lemma \ref{BetaaaBound},
	\begin{equation*}
	2\size{\L} \le (k-2)|\L| + 2\beta(\L)
	\end{equation*}
	Hence,
	\begin{align*}
	2\size{G} &= 2\size{\HH} + 2\size{\HH, \L} + 2\size{\L}\\
	&= 2\size{\HH} + 2((k-1)\card{\L} - 2\size{\L}) + 2\size{\L}\\
	&= 2\size{\HH} + 2(k-1)\card{\L} - 2\size{\L}\\
	&\ge 2\size{\HH} + k\card{\L} - 2\beta(\L),\\
	\end{align*}
	which is
	\begin{equation}
	\beta(\L) \ge \size{\HH} + \frac{k}{2}\card{\L} - \size{G}.
	\label{BetaBound}
	\end{equation}
	Let $M$ be the maximum of $\size{I, V(G) \setminus I}$ over all independent sets $I$ of $G$ with $I \subseteq \HH$. Then
	\begin{equation*}
		\mic(G) \ge M + (k-1)\beta(\L).
	\end{equation*}
	Applying Kernel Magic and using (\ref{BetaBound}) gives
	\begin{align*}
	2\size{G} &\ge (k-2)\card{G} + M + (k-1)\beta(\L) + 1\\
	&\ge (k-2)\card{G} + M + (k-1)\parens{\size{\HH} + \frac{k}{2}\card{\L} - \size{G}} + 1\\
	&= (k-2)\card{G} + M + (k-1)\size{\HH} + \frac{k(k-1)}{2}\card{\L} - (k-1)\size{G} + 1.
	\end{align*}		
	Hence
	\begin{equation}
	(k+1)\size{G} \ge (k-2)\card{G} + M + (k-1)\size{\HH} + \frac{k(k-1)}{2}\card{\L} + 1
	\label{KPOBound}
	\end{equation}
	Let $\C$ be the components of $G[\HH]$.  Then $\alpha(C) \ge \frac{\card{C}}{\chi(C)}$ for all $C \in \C$.  Whence
	\begin{equation}
	  M + (k-1)\size{\HH} \ge \sum_{C \in \C} k\frac{\card{C}}{\chi(C)} + (k-1)\size{C}.
	  \label{Mbound}
	\end{equation}
	
	If $\L = \emptyset$, then $G$ has average degree at least $k \ge k-1 + \frac{k-3}{k^2-2k+2}$.  So, assume $\L \ne \emptyset$.  Then $G[\HH]$ is $(k-1)$-colorable by $k$-list-criticality of $G$. In particular, $\chi(C) \le k-1$ for every $C \in \C$.
	For every $C \in \C$,
	\begin{equation}
	 k\frac{\card{C}}{\chi(C)} + (k-1)\size{C} \ge \parens{k - \frac12}\card{C}.
	 \label{KFC}
	\end{equation}
	To see this, first suppose $C \in \C$ is not a tree. Then $\size{C} \ge \card{C}$ and hence $k\frac{\card{C}}{\chi(C)} + (k-1)\size{C} \ge k\frac{\card{C}}{k-1} + (k-1)\card{C} \ge (k - \frac12)\card{C}$.  If $C$ is a tree, then $\chi(C) \le 2$ and hence 
	$k\frac{\card{C}}{\chi(C)} + (k-1)\size{C} \ge k\frac{\card{C}}{2} + (k-1)(\card{C} - 1) \ge (k-\frac12)\card{C}$ unless $\card{C} = 1$.  This proves (\ref{KFC}) since the bound is trivially satisfied when $\card{C} = 1$.
	
	Now combining (\ref{KPOBound}), (\ref{Mbound}) and (\ref{KFC}) with the basic bound
	\begin{equation*}
	  \card{\L} \ge k\card{G} - 2\size{G},
	\end{equation*}	
	gives
	\begin{align*}
	(k+1)\size{G} &\ge (k-2)\card{G} + \parens{k - \frac12}\card{\HH} + \frac{k(k-1)}{2}\card{\L} + 1\\
	&= \parens{2k-\frac52}\card{G} + \frac{k^2 - 3k + 1}{2}\card{\L} + 1\\
	&\ge\parens{2k-\frac52}\card{G} + \frac{k^2 - 3k + 1}{2}\parens{k\card{G} - 2\size{G}} + 1.\\
	\end{align*}
	After some algebra, this becomes
	\begin{equation*}
		2\size{G} \ge \parens{k-1 + \frac{k-3}{k^2 -2k + 2}}\card{G} + \frac{2}{k^2 -2k + 2}.
	\end{equation*}
	That proves the theorem.
\end{proof}

The right side of equation (\ref{KFC}) in the above proof can be improved to $k\card{C}$ unless $C$ is a $K_2$ where both vertices have degree $k$ in $G$.  
If these $K_2$'s could be handled, the average degree bound would improve to $k-1 + \frac{k-3}{(k-1)^2}$.  

\begin{TheConjecture}
Every incomplete (online) $k$-list-critical graph has average degree at least \[k-1 + \frac{k-3}{(k-1)^2}.\]
\end{TheConjecture}

\bibliographystyle{amsplain}
\bibliography{GraphColoring1}
\end{document}